\documentclass[11pt,a4paper]{amsart}

\usepackage[T1]{fontenc}
\usepackage{lmodern}
\usepackage{mathtools,amssymb}
\usepackage{microtype}
\usepackage{needspace}
\usepackage[a4paper,hmargin=30mm,vmargin=28mm]{geometry}
\usepackage[hidelinks]{hyperref}
\hypersetup{pdftitle={A question of Ulam and a theorem of de Rham},pdfauthor={Alexander Lytchak},
  pdfsubject={Product decompositions of complete metric spaces}}

\newtheorem{theorem}{Theorem}[section]
\newtheorem{proposition}[theorem]{Proposition}
\newtheorem{lemma}[theorem]{Lemma}
\newtheorem{corollary}[theorem]{Corollary}
\newtheorem{mainthm}{Theorem}

\theoremstyle{definition}

\theoremstyle{remark}

\numberwithin{equation}{section}
\newcommand{\R}{\mathbb R}
\newcommand{\Q}{\mathbb Q}
\newcommand{\N}{\mathbb N}
\newcommand{\Z}{\mathbb Z}
\newcommand{\T}{\mathbb T}
\newcommand{\one}{\mathbf 1}
\newcommand{\cP}{\mathcal P}
\newcommand{\cB}{\mathcal B}
\newcommand{\id}{\mathrm{id}}
\newcommand{\spn}{\operatorname{span}}

\newcommand{\im}{\operatorname{im}}

\allowdisplaybreaks[1]
\title{A question of Ulam and a theorem of de Rham}
\author{Alexander Lytchak}
\begin{document}
\begin{abstract}
	We generalize de Rham's decomposition theorem to proper metric spaces:
	every proper metric space has a unique decomposition as the product of
	a Euclidean space and a finite or countable pointed $\ell^2$-product
	of nontrivial irreducible non-Euclidean factors. A slightly weaker unique decomposition statement is verified for general complete metric spaces. We use this to  resolve a question of Ulam on the uniqueness of square roots of complete metric spaces.
\end{abstract}

\maketitle

\section{Introduction}
\subsection{Results}
In this paper, products of metric spaces carry the
$\ell^2$-product metric:
\begin{equation*}
 d\bigl((a,b),(a',b')\bigr)^2
   =d(a,a')^2+d(b,b')^2.
\end{equation*}

Stanis\l aw Ulam asked whether, for metric spaces $X$ and $Y$, the
existence of an isometry
\[
 X\times X\cong Y\times Y
\]
implies that $X$ and $Y$ are isometric; see
\cite[Problem~77(b)]{Mau81} and \cite{Ula60}.
He also asked analogous questions for topological spaces and groups;
see \cite[Problem~77(a), (c)]{Mau81}. Related decomposition questions
in other categories are discussed in \cite[Remark~1.1]{FL08}.

Fournier~\cite{Fou71} showed that $X=\Q$ and  $Y=\sqrt2\Q$, both with the metrics
inherited from  $\R$, provide  a negative answer to Ulam's question in general. 
The question was subsequently considered for complete and compact
metric spaces; see \cite[commentary on Problem~77(b)]{Mau81}. Our first result gives
an affirmative answer for proper metric spaces:

\begin{mainthm}\label{thm:roots}
Let $A,B$ be proper metric spaces and let $n$ be a positive integer. Then
\[
 A^n\cong B^n\quad\Longrightarrow\quad A\cong B.
\]
\end{mainthm}

Surprisingly, in the general complete setting the answer to Ulam's question is negative:

\begin{mainthm}\label{thm:roots2}
There exist closed subsets $A,B$ of a separable Hilbert space
$\mathbb H$ such that
\[
 A\times A\cong B\times B,\qquad A\not\cong B.
\]
\end{mainthm}

%Theorem~\ref{thm:roots2} uses a form of  Laczkovich's obstruction to measurable
%matching~\cite{Lac88}.  Informally speaking, spaces $A$ and $B$ are constructed  as measurable products   

An essential property of the counterexample is that  every  isometry  $F:A^2\to B^2$ in Theorem \ref{thm:roots2} disrespects all choices of base points.  More precisely, for every isometry $F$ the image $F(\Delta _A)$ of the diagonal $\Delta _A\subset A\times A$ is disjoint from $\Delta _B\subset B\times B$.  

  On the other hand, for pointed complete metric spaces, 
 Ulam's question has a positive answer.
We write $\cong_*$ for a basepoint-preserving surjective isometry.

\begin{mainthm}\label{thm:roots3}
Let $(A,a)$ and $(B,b)$ be complete pointed metric spaces and let $n$
be a positive integer. Then
\[
 (A^n,(a,\ldots,a))\cong_*(B^n,(b,\ldots,b))
 \quad\Longrightarrow\quad
 (A,a)\cong_*(B,b).
\]
\end{mainthm}

%In particular, for the spaces in Theorem~\ref{thm:roots2}, no isometry
%of their squares sends a diagonal point to a diagonal point. For Banach
%spaces, translations remove this obstruction; hence
%Theorem~\ref{thm:roots3} also gives uniqueness of finite roots in that
%category, with the $\ell^2$-product metric.

The affirmative results follow from decomposition theorems
generalizing a classical theorem of de Rham.
We call a metric space \emph{irreducible}
if it has more than one point and every decomposition into two factors
has a one-point factor.
The theorem of de Rham~\cite{dR52} together with its extension by Eschenburg and Heintze~\cite{EH98} shows that a complete connected Riemannian manifold has a unique decomposition as
a product of a Euclidean space and non-Euclidean irreducible
Riemannian manifolds. Analogous results were obtained for compact
subsets of Euclidean spaces by Moszy\'nska~\cite{Mos92}, for
finite metric spaces by Avgustinovich and Fon-Der-Flaass~\cite{AF00}, for
geodesic metric spaces of finite  dimension by Foertsch and
Lytchak~\cite[Theorem~1.1]{FL08}, and for some CAT(0) spaces by
Caprace and Monod~\cite{CM09}. Here we prove existence
and uniqueness for arbitrary proper metric spaces, with no
geodesicity, curvature, or dimension assumptions.

\begin{mainthm}\label{thm:proper}
	Every proper metric space $X$ admits a decomposition
	\[
	X\cong\R^m\times\prod_{i\in I}(X_i,o_i),
	\]
	where $m\ge0$, $I$ is finite or countable, and the $X_i$ are nontrivial
	irreducible proper metric spaces not isometric to $\R$. For each
	basepoint, the Euclidean factor fiber and the $X_i$-fibers are uniquely
	determined, up to permutation of the $X_i$-fibers.
\end{mainthm}

Here and below, infinite products mean pointed $\ell^2$-products; see
Subsection~\ref{subsec:infinite-products}.
The uniqueness statement in Theorem~\ref{thm:proper} is much stronger
than uniqueness of the abstract factors $X_i$ up to isometry and
permutation: their fibers through a fixed basepoint are uniquely
determined as subsets of $X$.
In particular, every isometry of $X$ acts by an isometry of $\R^m$
and, up to a permutation $\sigma$ of $I$, by isometries
$\phi_i:X_i\to X_{\sigma(i)}$. For an infinite pointed product, these
coordinate isometries must satisfy
\[
 \sum_{i\in I}d\bigl(\phi_i(o_i),o_{\sigma(i)}\bigr)^2<\infty.
\]
Conversely, any such data define an isometry of $X$.

Euclidean spaces and, more generally,
Hilbert spaces play a special role. Although $\R^n$ is a product of
$n$ irreducible factors, each isometric to $\R$, the factor fibers
through a fixed basepoint are not uniquely determined when $n\ge2$.
Recently, Foertsch, Lytchak, and Soultanis~\cite[Theorem~1.1]{FLS25}
proved that every complete metric space has a canonical decomposition
$X\cong \mathbb H\times X'$, where $\mathbb H$ is a maximal Hilbert space factor and $X'$
does not admit $\R$ as a direct factor; see also \cite{AGP06} for the case of Banach spaces. We show that Hilbert factors
are the only possible source of nonuniqueness of irreducible factor
fibers in complete metric spaces; see
Theorem~\ref{thm:canonical-decomposition}.

A  complete metric space need not have any nontrivial 
irreducible factors, see
Lenze~\cite[Theorem~B]{Len25}.  Such a  failure is essential for the non-uniqueness  in Theorem \ref{thm:roots2}.

 The central result of the paper, which is the main ingredient in the
affirmative results above, gives common refinements for complete
metric spaces without line factors.
\begin{mainthm}\label{thm:refinement}
Suppose that $X$ is a bounded metric space or a complete metric space that does not admit
$\R$ as a direct factor. Let
\[
X=A\times B=C\times D
\]
be two product decompositions, and identify the four factors with their
fibers through a common point $o$. Then
\begin{equation}\label{eq:main-refinement}
	X\cong(A\cap C)\times(A\cap D)\times(B\cap C)\times(B\cap D).
\end{equation}
 In particular, $
 A=(A\cap C)\times(A\cap D).$
\end{mainthm}

\subsection{Ideas} Known proofs of de Rham-type decomposition theorems reduce geometric
questions to algebraic structures associated with the space. In the
Riemannian setting, the proofs in~\cite{dR52,EH98} use holonomy
and, in the non-simply connected case, the action of the deck
transformation group. The proof in~\cite{FL08} uses
the structure of maximal subsets isometric to convex sets in
normed spaces.

In the present paper, the proof also relies on some hidden algebra.
Somewhat surprisingly, even without any additional structure on the
metric space, product decompositions turn out to be related to a
sufficiently rich algebraic structure.
Our starting point relates uniqueness of product decompositions---more
precisely, the existence of common refinements as in
Lemma~\ref{lem:geometric-refinement}---to commutation of the associated
factor projections. We regard these projections as
$1$-Lipschitz self-maps of $X$, projecting $X$ onto a factor fiber through
a fixed basepoint $o\in X$. This reduces the geometric question of the
relative position of factors to commutation of such self-maps.

The relation between commuting projections and common refinements is
classical in Banach-space theory. In particular, for each fixed
$p\in[1,\infty]\setminus\{2\}$, all $L^p$-projections on a Banach space
commute~\cite{Beh76}; see also~\cite[p.~45]{HWW93}. The exponent
$p=2$, which is the only exponent considered for metric products in this
paper, is exceptional: orthogonal projections on a Hilbert space of
dimension at least two need not commute.

The next step in the algebraization program relates commutation in the
semigroup of self-maps to linear algebra. We describe an appropriate
vector subspace $V=V(X)\subset\R^X$ of the space of functions on $X$
that separates points of $X$ and is invariant under the
pullbacks $P\cdot f=f\circ p$ of all factor projections at $o$.
Once such a vector space is defined, commutation of the metric
projections $p,q:X\to X$ is equivalent to commutation of the induced
linear projections $P,Q:V\to V$.

The form of the product metric suggests a choice of $V$. We consider
the vector space generated by the generalized scalar-product functions
\[
 k_x(z)=\frac12\bigl(d(o,x)^2+d(o,z)^2-d(x,z)^2\bigr).
\]
These functions satisfy $k_x\circ p=k_{p(x)}$ for every factor projection $p$ at $o$. Therefore,
their linear span $V$ is invariant under pullback by all such
projections and separates points of $X$.
This space  $V$ carries a  bilinear form $\beta$ with $\beta (k_x,k_y)=k_x(y)$; it  appears
in several contexts, including Schoenberg's criterion for isometric
embeddability into Hilbert spaces~\cite{Sch38}.

The vector space $V$ and the bilinear form $\beta$ are compatible
with product decompositions. The crucial additional ingredient is a
functional on $V$ that is additive across the factors.
For bounded spaces $X$, one can consider the oscillation norm.   It relates  product decompositions  of a bounded metric space $X$  to the well-investigated    $\ell^1$-products of normed space  $V(X)$, which quickly leads to the proof of Theorem \ref{thm:refinement} in this case.

For
unbounded $X$,  the choice of the functional is  more involved and the proof of Theorem \ref{thm:refinement} is more technical.  This is not surprising, since  the proof has to recognize the existence of a line factor,  based only on the existence of a pair of non-commuting projections.
We use a convex functional $\Phi$ resembling the squared norm on a Hilbert space. The   presence of non-commuting projections in $X$ produces  a special \emph{quadratic direction} $h$ for the functional $\Phi$, so that the restriction of $\Phi$ to any affine line in the direction of $h$ is a quadratic function.  Then we show that such  a quadratic direction $h\in V$ corresponds to a line factor of $X$, given that $X$ is complete.

%These ideas lead to the refinement theorem, Theorem~\ref{thm:refinement}.
The existence of a decomposition in Theorem~\ref{thm:proper} also
relies on algebra.
The key step here is to recognize that, once commutation is established,
the space of factor projections at $o$ is a Boolean algebra in a natural
way. Such algebras have a surprisingly rich structure, and their
properties translate into properties of the underlying metric spaces.

Theorem~\ref{thm:roots} follows from Theorem~\ref{thm:proper} by
matching the irreducible factors of two $n$-th roots. For pointed complete spaces,
Theorem~\ref{thm:roots3} uses a further algebraic tool: the
\emph{cardinal algebras} introduced by Tarski~\cite{Tar49}; see also
\cite[Section~2]{KM16}. More precisely,  the
pointed isometry classes of complete metric spaces without line
factors form a cardinal algebra, once we uniformly bound the cardinalities of the underlying spaces. Theorem~\ref{thm:refinement} turns out to be sufficient to  verify the axioms of a cardinal algebra required in \cite{Tar49}. Then  Theorem \ref{thm:roots3}   follows
from Tarski's cancellation theorem~\cite[Theorem~2.35]{Tar49}.

For Theorem~\ref{thm:roots2}, we start with Laczkovich's measurable bipartite graph without a measurable perfect matching~\cite{Lac88}. The connected components of the graph are bi-infinite paths. Assigning distinct lengths in $(1,2)$ to the edges, we turn the vertex sets of these components into rigid metric spaces, isometric to discrete subsets of $\mathbb R$, such that different components are nonisometric.

The spaces $A,B$ in Theorem~\ref{thm:roots2} are defined as $L^2$-products (see \cite[Remark 45]{Monod}, \cite{AH}) of these components over the two vertex classes $V_0$ and $V_1$, with each component based at the corresponding vertex. The rigidity of the components implies that an isometry between $A$ and $B$ would induce a measurable bijection between $V_0$ and $V_1$, preserving components and having integrable displacement. A strengthening of Laczkovich's matching obstruction rules out such a measurable bijection.

On the other hand, each vertex is incident to exactly two edges. This allows us to identify both squares $A^2$ and $B^2$ with the $L^2$-product of the same family of components over the edge space.

The paper is organized as follows. Section~\ref{sec:commutation}
derives common refinement and uniqueness from commutation of factor
projections, and Section~\ref{sec:functions} constructs the linearization.
Sections~\ref{sec:bounded} and~\ref{sec:complete} prove commutation for
bounded spaces and for complete spaces without line factors, respectively.
Section~\ref{sec:projections} studies the resulting Boolean algebra,
constructs its atomic and atomless factors, and proves atomicity under
properness. Section~\ref{sec:decompositions} establishes the canonical
decomposition and proves Theorems~\ref{thm:proper} and~\ref{thm:roots}.
Section~\ref{sec:pointed-roots} constructs the cardinal algebra and
proves Theorem~\ref{thm:roots3}. Section~\ref{sec:counterexample} proves
Theorem~\ref{thm:roots2}.

\subsection*{Acknowledgements}
This paper arose in close collaboration with ChatGPT, an AI assistant
developed by OpenAI. The collaboration involved exploring proof strategies,
refining arguments, locating references, and drafting and revising the
manuscript. The author takes full responsibility for the mathematical
content and the final text.

 The author was supported in part by the German Research Foundation under Germany’s Excellence Strategy – EXC-2047/1 – 390685813.

\section{Commutation and uniqueness}
\label{sec:commutation}

Throughout the paper, metric spaces are assumed to be nonempty, and
an empty product is a one-point space.
In this section, $X$ is an arbitrary metric space with a fixed point
$o$. 

\subsection{Factor projections}
If $X=A\times B$ and $o=(o_A,o_B)$, we identify $A$ with
$A\times\{o_B\}$ and $B$ with $\{o_A\}\times B$. The associated
factor projections are
\[
 p(a,b)=(a,o_B),\qquad \bar p(a,b)=(o_A,b).
\]
They are $1$-Lipschitz retractions, fix $o$, and satisfy
\begin{equation*}
 d(x,y)^2=d(p(x),p(y))^2+d(\bar p(x),\bar p(y))^2.
\end{equation*}

The map $p$ is the unique nearest-point projection onto its image:
for $a\in A$,
\[
 d(x,a)^2=d(p(x),a)^2+d(x,p(x))^2.
\]
In particular, a \emph{factor projection} is determined by its image.
Its complementary factor is
$\bar p(X)=p^{-1}(o)$, so the complementary projection is determined
by $p$ as well. We allow the identity and the constant map $x\mapsto o$
as factor projections.

\subsection{Commutation and refinement}
The following lemma gives the geometric consequence of commutation.

\begin{lemma}\label{lem:geometric-refinement}
Let $X=A\times B=C\times D$ be two decompositions of a metric space,
with factors identified with their fibers through $o$. Suppose that
the four corresponding factor projections $p,\bar p,q,\bar q$ onto
$A,B,C,D$, respectively, commute pairwise. Then the map $(q,\bar q)$
restricts to an isometry
\begin{equation}\label{eq:qbarq}
 (q,\bar q):A\longrightarrow(A\cap C)\times(A\cap D),
\end{equation}
and the map with components
$p\circ q,p\circ\bar q,\bar p\circ q,\bar p\circ\bar q$
defines a common refinement:
\begin{equation}\label{eq:ref}
 X=(A\cap C)\times(A\cap D)\times(B\cap C)\times(B\cap D).
\end{equation}
\end{lemma}

\begin{proof}
For any point $a\in A$, we have
\[
 q(a)=q(p(a))=p(q(a))\in A\cap C,
\]
and similarly $\bar q(a)\in A\cap D$. Thus
$(q,\bar q):X\to C\times D$ restricts to an isometric embedding
\[
 A\longrightarrow(A\cap C)\times(A\cap D).
\]
Conversely, for any $a_C\in A\cap C$ and $a_D\in A\cap D$, the
decomposition $X=C\times D$ gives a point $a\in X$ with
$q(a)=a_C$ and $\bar q(a)=a_D$. Commutation gives
\[
 q(p(a))=p(q(a))=a_C,\qquad \bar q(p(a))=p(\bar q(a))=a_D.
\]
Thus $p(a)$ and $a$ have the same coordinates in $C\times D$, so
$a=p(a)\in A$. Hence the isometric embedding in \eqref{eq:qbarq}
is surjective.

This proves that $A$ decomposes as the direct product of $A\cap C$
and $A\cap D$. The same argument applies to $B$. Combining these
splittings with $X=A\times B$ gives \eqref{eq:ref}. Interchanging
the two original decompositions gives the analogous splittings of
$C$ and $D$, so this is a common refinement.
\end{proof}

\subsection{Uniqueness of decompositions}
\label{subsec:infinite-products}
For pointed metric spaces $(X_i,o_i)$, write
\[
 \prod_{i\in I}(X_i,o_i)
 =\left\{(x_i):\sum_{i\in I}d(x_i,o_i)^2<\infty\right\},
\]
with its usual $\ell^2$-metric and basepoint $(o_i)_{i\in I}$. The sum over
an arbitrary index set is the supremum of its finite subsums.

\begin{corollary}\label{cor:commuting-uniqueness}
Let $X$ be a metric space with basepoint $o$, and suppose that all
factor projections at $o$ commute. If, for an arbitrary index set $I$,
$X$ admits a decomposition
\[
 X\cong\prod_{i\in I}(X_i,o_i)
\]
into nontrivial irreducible metric spaces, with $o$ corresponding to
$(o_i)_{i\in I}$, then its factor fibers through $o$ are uniquely
determined up to permutation.
\end{corollary}

\begin{proof}
The one-point case is immediate. Compare two decompositions
$X=\prod_{i\in I}(X_i,o_i)=\prod_{j\in J}(Y_j,o'_j)$ into nontrivial
irreducible factors, and identify all factors with their fibers
through $o$. Let $p_i,q_j$ be the corresponding coordinate projections,
regarded as self-maps of $X$. Lemma~\ref{lem:geometric-refinement},
applied to each pair of binary splittings, gives
\[
 X_i=(X_i\cap Y_j)\times(X_i\cap Y_j^{\perp}),
\]
where $Y_j^{\perp}$ is the complementary factor fiber. Irreducibility
of $X_i$ implies that $q_j$ restricts to either the identity or the
constant map $o$ on $X_i$.

Fix $i\in I$ and choose $x\in X_i\setminus\{o\}$. Since the
$Y_j$-coordinates determine $x$, there is some $j\in J$ with
$q_j(x)\ne o$. Hence $q_j$ is the identity on $X_i$, so
$X_i\subseteq Y_j$. Since $p_i$ fixes the nontrivial subset
$X_i\subseteq Y_j$, the same argument applied to $Y_j$ shows that
$p_i$ is the identity on $Y_j$. Thus $Y_j\subseteq X_i$, and
$X_i=Y_j$.

This $j$ is unique, because distinct factor fibers of the second
decomposition intersect only at $o$. Reversing the roles of the two
decompositions shows that every $Y_j$ occurs. Thus there is a
bijection $\sigma:I\to J$ with $X_i=Y_{\sigma(i)}$ for every $i\in I$.
\end{proof}

\section{Linearization}
\label{sec:functions}

Let $X$ be a metric space with basepoint $o$, and write $r(x)=d(o,x)$.

\subsection{The function space}
Define the \emph{generalized scalar-product  kernel} by
\begin{equation}\label{eq:kernel}
 k(x,z)=k_x(z)=\frac12\bigl(r(x)^2+r(z)^2-d(x,z)^2\bigr),
 \qquad x,z\in X,
\end{equation}
and let
\[
 V=\spn_{\R}\{k_x:x\in X\}.
\]
All functions in $V$ vanish at $o$. The triangle inequality gives
\begin{equation}\label{eq:kernel-growth}
 |k_x(z)|\le r(x)r(z).
\end{equation}
Thus every $f\in V$ has at most linear growth. By definition, $
 k_x(y)=k_y(x).$
 
The following construction plays a central role in
Schoenberg's theorem~\cite{Sch38}.
See also the quasilinearization of Berg and Nikolaev~\cite{BN08}.

\begin{lemma}\label{lem:pairing}
There is a well-defined symmetric bilinear form $\beta$
on $V$ such that
\begin{equation}\label{eq:pairing-evaluation}
 \beta(f,k_x)=f(x)\qquad(f\in V,\ x\in X).
\end{equation}
It is nondegenerate and satisfies
\begin{equation}\label{eq:pairing-distance}
 \beta(k_x,k_x)=r(x)^2,\qquad
 \beta(k_x-k_y,k_x-k_y)=d(x,y)^2.
\end{equation}
In particular, $x\mapsto k_x$ is injective, and $V$ separates the
points of $X$.
\end{lemma}

\begin{proof}
Prescribe $\beta(k_x,k_y)=k_x(y)$ and extend by bilinearity.
To check that this defines a bilinear form on $V$, suppose that
$\sum_i a_i k_{x_i}=0$ in $V$. Then, for every $y\in X$,
\[
 \sum_i a_i k_{x_i}(y)=\sum_i a_i k_y(x_i)=0.
\]
Thus every linear relation among the generators pairs to zero with
every generator. The bilinear extension is
therefore well-defined, and it is symmetric because $k$ is
symmetric.

The distance identities \eqref{eq:pairing-distance} follow directly
from \eqref{eq:kernel}. Both sides of
\eqref{eq:pairing-evaluation} are linear in $f$ and agree for every
$f=k_y$, $y\in X$. Hence \eqref{eq:pairing-evaluation} holds for all
$f\in V$.

If $\beta(f,g)=0$ for every $g\in V$, then
\eqref{eq:pairing-evaluation} gives $f(x)=0$ for every $x\in X$.
Thus $\beta$ is nondegenerate. Injectivity of $x\mapsto k_x$ follows
from \eqref{eq:pairing-distance}.

For $x\ne y$, the function $f=k_x-k_y$ satisfies
\[
 f(x)-f(y)=\beta(k_x-k_y,k_x-k_y)=d(x,y)^2>0.
\]
Thus $V$ separates points of $X$.
\end{proof}

By Schoenberg's theorem, $\beta$ is positive definite if and only if
$X$ admits an isometric embedding into a Hilbert space.

\subsection{Geometric and algebraic splitting}
We denote the identity map on any vector space by $I$.

\begin{lemma}\label{lem:linear-product-decomposition}
Let $X=A\times B$ be a product decomposition with basepoint
$o=(o_A,o_B)$ and complementary factor projections $p,\bar p$.
Then the following hold.
\begin{enumerate}
\item For every $f\in V$, the functions $f_A:=f\circ p$ and
$f_B:=f\circ\bar p$ belong to $V$.
\item Every $f\in V$ satisfies $f=f_A+f_B$. In particular,
\[
 f(a,b)=f(a,o_B)+f(o_A,b)
 \qquad((a,b)\in A\times B).
\]
\item The map $P:V\to V$ defined by
\begin{equation*}
 P\cdot f=f\circ p
\end{equation*}
is a linear projection symmetric with respect to $\beta$.
Its complementary projection satisfies $(I-P)\cdot f=f\circ\bar p$.
\end{enumerate}
\end{lemma}

\begin{proof}
The product formula for squared distances gives
\begin{equation*}
 k_x\circ p=k_{p(x)},\qquad
 k_x\circ\bar p=k_{\bar p(x)},\qquad
 k_x=k_{p(x)}+k_{\bar p(x)}.
\end{equation*}
Since the functions $k_x$ span $V$, the first two assertions follow
by linearity.

The map $P$ is linear, and $P^2=P$ follows from $p\circ p=p$.
The formula for $I-P$ follows from the second assertion.
Finally, for every $f\in V$ and $x\in X$, the evaluation identity
\eqref{eq:pairing-evaluation} gives
\begin{align*}
 \beta(P\cdot f,k_x)
 &=(P\cdot f)(x)=f(p(x))\\
 &=\beta(f,k_{p(x)})=\beta(f,P\cdot k_x).
\end{align*}
Since the functions $k_x$ span $V$, this proves
\begin{equation*}
 \beta(P\cdot f,g)=\beta(f,P\cdot g)\qquad(f,g\in V).\qedhere
\end{equation*}
\end{proof}

\subsection{Commutation of geometric and algebraic projections}
\label{subsec:geometric-algebraic-commutation}
Let $p,q$ be factor projections at $o$ associated with product
decompositions $X=A\times B=C\times D$, and let $P,Q$ be the induced
linear projections. For every $x\in X$,
\[
 PQ\cdot k_x=k_{p(q(x))},\qquad
 QP\cdot k_x=k_{q(p(x))}.
\]
Since the $k_x$ span $V$ and $x\mapsto k_x$ is injective,
\[
 p\circ q=q\circ p\quad\Longleftrightarrow\quad PQ=QP.
\]

\section{Commutation in bounded spaces}
\label{sec:bounded}

The proof in this section uses only boundedness. Its linear-algebraic
ingredient is the standard commutation theorem for $L$-projections,
that is, linear projections on normed spaces satisfying the identities
in Lemma~\ref{lem:L-projections} below;
see~\cite{Cun60} and~\cite[Theorem~I.1.10(a)]{HWW93}. We recall the
elementary proof, which does not use completeness.

\begin{lemma}\label{lem:L-projections}
Let $P,Q$ be linear projections on a normed space $E$ such that
\[
 \|v\|=\|P\cdot v\|+\|(I-P)\cdot v\|
      =\|Q\cdot v\|+\|(I-Q)\cdot v\|\qquad(v\in E).
\]
Then $PQ=QP$.
\end{lemma}

\begin{proof}
For $v\in\im P$, we have
$(I-P)(I-Q)\cdot v=-(I-P)Q\cdot v$. Applying the norm identities gives
\begin{align*}
 \|v\|
 &=\|PQ\cdot v\|+\|(I-P)Q\cdot v\|
   +\|P(I-Q)\cdot v\|+\|(I-P)(I-Q)\cdot v\|\\
 &=\|PQ\cdot v\|+\|P(I-Q)\cdot v\|+2\|(I-P)Q\cdot v\|\\
 &\ge\|(PQ+P(I-Q))\cdot v\|+2\|(I-P)Q\cdot v\|\\
 &=\|v\|+2\|(I-P)Q\cdot v\|.
\end{align*}
Hence $(I-P)Q\cdot v=0$, so $Q\cdot v$ belongs to
$\ker(I-P)=\im P$. Thus $Q$ preserves $\im P$.
Replacing $P$ by $I-P$ shows that $Q$ also preserves $\ker P$,
and therefore $PQ=QP$.
\end{proof}

\begin{proposition}\label{prop:bounded-commutation}
For every bounded metric space $X$ and $o\in X$, the factor projections
at $o$ commute.
\end{proposition}

\begin{proof}
All functions in $V$ are bounded, and the oscillation seminorm
\begin{equation*}
 \|f\|_{\mathrm{osc}}=\sup_X f-\inf_X f
\end{equation*}
is a norm on $V$, since every $f\in V$ vanishes at $o$.

For a splitting $X=A\times B$, let $P,\bar P=I-P$ be the induced
linear projections. The identity $f=P\cdot f+\bar P\cdot f$ and
Lemma~\ref{lem:linear-product-decomposition}(2) imply
\begin{align*}
 \sup_X f&=\sup_X(P\cdot f)+\sup_X((I-P)\cdot f),\\
 \inf_X f&=\inf_X(P\cdot f)+\inf_X((I-P)\cdot f).
\end{align*}
Subtracting these identities yields
\begin{equation*}
 \|f\|_{\mathrm{osc}}
 =\|P\cdot f\|_{\mathrm{osc}}+\|(I-P)\cdot f\|_{\mathrm{osc}}.
\end{equation*}
Thus every induced projection $P$ is an $L$-projection on the same normed
space $(V,\|\cdot\|_{\mathrm{osc}})$. By Lemma~\ref{lem:L-projections},
all these linear projections commute. The corresponding factor
projections on $X$ therefore commute as well, as explained in
Subsection~\ref{subsec:geometric-algebraic-commutation}.
\end{proof}

Combining Proposition~\ref{prop:bounded-commutation}
and Lemma~\ref{lem:geometric-refinement}, we deduce:
\begin{corollary}\label{cor:bounded}
	Any two product decompositions of a bounded metric space have the
	common refinement \eqref{eq:main-refinement}.
\end{corollary}

\section{Commutation in complete spaces}
\label{sec:complete}

We prove Theorem~\ref{thm:refinement} by showing that noncommuting
factor projections produce a splitting line. Let $X$ be an arbitrary
metric space with basepoint $o$; completeness will be assumed only
in Proposition \ref{prop:kernel-recognition} and Lemma \ref{lem:realization}. We retain $V$, $\beta$, and
$r(x)=d(o,x)$ from Section~\ref{sec:functions}. The one-point case is
immediate, so we assume that $X$ has at least two points.

\subsection{A convex functional}
\label{subsec:functional}

In the absence of boundedness, the oscillation norm need not be finite.
We replace it by a convex functional whose additivity under product
splittings survives. By \eqref{eq:kernel-growth}, every $f\in V$ has
at most linear growth, so
\begin{equation}\label{eq:growth-norm}
 \|f\|_*:=\sup_{z\ne o}\frac{|f(z)|}{r(z)}
\end{equation}
is a finite norm on $V$. This auxiliary norm records the optimal growth
rate of $f$ relative to $o$ over all scales. We define the functional
$\Phi$ on $V$ by
\begin{equation}\label{eq:Phi}
 \Phi(f)=\sup_{z\in X}\bigl(2f(z)-r(z)^2\bigr).
\end{equation}

The functional $\Phi$ is reminiscent of Fitzpatrick functions associated
with monotone operators; see \cite{Fit88}.  If $X=\mathbb H$ is a Hilbert space, then $V$ is the dual space of $\mathbb H$,  and $\Phi$ is the square of the Hilbert-space norm.

\begin{lemma}\label{lem:Phi-basic}
The functional $\Phi$ is finite, convex, and locally Lipschitz with
respect to the norm \eqref{eq:growth-norm}. Moreover,
\begin{equation}\label{eq:Phi-bound}
 0\le\Phi(f)\le\|f\|_*^2\qquad(f\in V).
\end{equation}
For every $x\in X$,
\begin{equation}\label{eq:Phi-kernel}
 \Phi(k_x)=r(x)^2,
\end{equation}
and the supremum defining $\Phi(k_x)$ is attained uniquely at $x$.
\end{lemma}

\begin{proof}
Taking $z=o$ gives $\Phi(f)\geq 0$. On the other hand, for every $z\in X$,
\[
 2f(z)-r(z)^2\le2\|f\|_*r(z)-r(z)^2\le\|f\|_*^2
\]
by \eqref{eq:growth-norm}. Thus $\Phi$ is finite and
\eqref{eq:Phi-bound} holds.

For each fixed $z\in X$, the map $f\mapsto 2f(z)-r(z)^2$ is affine
on $V$. Thus $\Phi$ is a supremum of affine functions and hence convex.

If $\|f\|_*,\|g\|_*\le M$, the suprema defining $\Phi(f)$ and
$\Phi(g)$ may be restricted to $r(z)\le2M$: outside this ball the
expressions being maximized are nonpositive, whereas their values
at $o$ are zero. Consequently,
\[
 |\Phi(f)-\Phi(g)|
 \le\sup_{r(z)\le2M}2|f(z)-g(z)|
 \le4M\|f-g\|_*.
\]
Finally, the kernel definition gives
\begin{equation} \label{eq:kernel-supremum-identity}
 2k_x(z)-r(z)^2=r(x)^2-d(x,z)^2,
\end{equation}
which proves \eqref{eq:Phi-kernel} and uniqueness of the maximizer.
\end{proof}

Additivity under products replaces the $L$-projection property used
in the bounded case.

\begin{lemma}\label{lem:additivity}
For every factor projection $p$ and its induced linear projection $P$,
\begin{equation}\label{eq:Phi-splitting}
 \Phi(f)=\Phi(P\cdot f)+\Phi((I-P)\cdot f)\qquad(f\in V).
\end{equation}
\end{lemma}

\begin{proof}
Identify $X=A\times B$ and $o=(o_A,o_B)$. By
Lemma~\ref{lem:linear-product-decomposition},
\[
 f(a,b)=f(a,o_B)+f(o_A,b),\qquad
 r(a,b)^2=r(a,o_B)^2+r(o_A,b)^2.
\]
The coordinates vary independently, so the supremum in \eqref{eq:Phi}
is the sum of the two factor suprema $\Phi(P\cdot f)$ and
$\Phi((I-P)\cdot f)$.
\end{proof}

The additivity of $\Phi$ implies a rectangle identity:
\begin{corollary}
With the notation above, for all $f,g\in V$ we have
\begin{equation}\label{eq:quadrangle}
 \Phi(f+g)+\Phi(f)
 =\Phi(f+P\cdot g)+\Phi(f+(I-P)\cdot g).
\end{equation}
\end{corollary}

\begin{proof}
By \eqref{eq:Phi-splitting}, both sides of \eqref{eq:quadrangle} equal
\[
 \Phi(P\cdot(f+g))+\Phi((I-P)\cdot f)
 +\Phi(P\cdot f)+\Phi((I-P)\cdot(f+g)).\qedhere
\]
\end{proof}

We next establish directional differentiability of $\Phi$ at the functions $k_x$.

\begin{lemma}\label{lem:quadratic-estimate}
For every $x\in X$ and $g\in V$,
\begin{equation}\label{eq:derivative-at-kernel}
 \lim_{t\to0}\frac{\Phi(k_x+t g)-\Phi(k_x)}{t}=2g(x).
\end{equation}
\end{lemma}

\begin{proof}
Fix $x\in X$ and consider the convex function
\begin{equation*}
 \Psi_x(g):=\Phi(k_x+g)-\Phi(k_x)-2g(x).
\end{equation*}
By \eqref{eq:Phi} and \eqref{eq:kernel-supremum-identity},
\begin{equation*}
 \Psi_x(tg)
 =\sup_{z\in X}\bigl(2t(g(z)-g(x))-d(x,z)^2\bigr)\ge0,
\end{equation*}
where the inequality follows by taking $z=x$. We will show that
$\Psi_x(tg)=O(t^2)$, first for the generators $k_y$ and then, by
convexity, for every $g\in V$.

Fix $y\in X$ and set $A=r(x)+d(x,y)$. For $z\in X$, put
$\rho=d(x,z)$. The kernel formula and the triangle inequality give
\begin{equation*}
 \begin{aligned}
  |k_y(z)-k_y(x)|
  &\le\tfrac12\bigl|r(z)^2-r(x)^2\bigr|
     +\tfrac12\bigl|d(y,z)^2-d(y,x)^2\bigr|\\
  &\le\tfrac{\rho}{2}\bigl(2r(x)+\rho+2d(x,y)+\rho\bigr)
   =A\rho+\rho^2.
 \end{aligned}
\end{equation*}
For $|t|\le1/4$, it follows that
\[
 2t\bigl(k_y(z)-k_y(x)\bigr)-\rho^2
 \le 2A|t|\rho-\tfrac12\rho^2
 \le 2A^2t^2.
\]
Taking the supremum over $z$ yields
$0\le\Psi_x(tk_y)\le2A^2t^2$.

Now write $g=\sum_{i=1}^m a_i k_{y_i}$. Since
$tg=m^{-1}\sum_{i=1}^m ma_i t k_{y_i}$, convexity gives
\[
 0\le\Psi_x(tg)
 \le\frac1m\sum_{i=1}^m\Psi_x(ma_i t k_{y_i})
 =O(t^2).
\]
This estimate holds for both signs of $t$. Hence
$\Psi_x(tg)/t\to0$ as $t\to0$, proving
\eqref{eq:derivative-at-kernel}.
\end{proof}

We next characterize the embedded copy of $X$ using $\Phi$ and $\beta$.

\begin{proposition}\label{prop:kernel-recognition}
Suppose that $X$ is complete. An element $w\in V$ equals $k_z$ for
some $z\in X$ if and only if
\begin{equation}\label{eq:supporting-inequality}
 \Phi(f)+\Phi(w)\ge2\beta(f,w)\qquad(f\in V).
\end{equation}
\end{proposition}

\begin{proof}
If $w=k_z$, then \eqref{eq:Phi} and \eqref{eq:Phi-kernel} give
\[
 \Phi(f)+\Phi(w)\ge2f(z)=2\beta(f,w)\qquad(f\in V).
\]
Conversely, suppose that $w$ satisfies \eqref{eq:supporting-inequality}.
We first show that the supremum
\[
 \Phi(w)=\sup_{y\in X}\bigl(2w(y)-r(y)^2\bigr)
\]
is attained at some $z\in X$. For this point, we then verify
\begin{equation}\label{eq:beta}
 \beta(g,w)=g(z)=\beta(g,k_z)\qquad(g\in V).
\end{equation}
By the nondegeneracy of $\beta$ in Lemma~\ref{lem:pairing}, this will
imply $w=k_z$.

To find such a point $z$, consider the defect
\[
 \delta(a):=\Phi(w)-2w(a)+r(a)^2\ge0\qquad(a\in X).
\]
By the definition of $\Phi(w)$, we have $\inf_X\delta=0$, so we may
choose a sequence $(z_n)$ with $\delta(z_n)\to0$. By completeness of
$X$ and continuity of $\delta$, it suffices to show that $(z_n)$ is
Cauchy. This follows  once the next estimate has been verified:
\begin{equation}\label{eq:ab}
 \delta(a)+\delta(b)\ge\frac12d(a,b)^2\qquad(a,b\in X).
\end{equation}
To prove \eqref{eq:ab}, we apply \eqref{eq:supporting-inequality} with
$f=(k_a+k_b)/2$ and obtain
\begin{align*}
 w(a)+w(b)-\Phi(w)
 \le\Phi\!\left(\frac{k_a+k_b}{2}\right)
 =\sup_{y\in X}\bigl((k_a+k_b)(y)-r(y)^2\bigr).
\end{align*}
Using \eqref{eq:kernel-supremum-identity}, we can rewrite this supremum as
\begin{align*}
 \frac12\sup_{y\in X}
   \bigl(r(a)^2-d(a,y)^2+r(b)^2-d(b,y)^2\bigr)=\frac12\Bigl(r(a)^2+r(b)^2
   -\inf_{y\in X}\bigl(d(a,y)^2+d(b,y)^2\bigr)\Bigr).
\end{align*}
Since $d(a,b)^2\le2d(a,y)^2+2d(b,y)^2$ for every $y\in X$, we obtain
\[
 2\bigl(w(a)+w(b)-\Phi(w)\bigr)
 \le r(a)^2+r(b)^2-\frac12d(a,b)^2.
\]
This is precisely \eqref{eq:ab}. Thus $(z_n)$ is Cauchy, and its
limit $z\in X$ satisfies $\delta(z)=0$.

To verify \eqref{eq:beta}, fix $g\in V$ and apply
\eqref{eq:supporting-inequality} with $f=k_z+tg$, where $t\in\R$.
Since $\delta(z)=0$, we obtain
\[
 \Phi(k_z+tg)+2w(z)-r(z)^2
 \ge2\beta(k_z+tg,w)=2w(z)+2t\beta(g,w).
\]
Using $r(z)^2=\Phi(k_z)$, we conclude that
\begin{equation}\label{eq:last}
 \Phi(k_z+tg)-\Phi(k_z)\ge2t\beta(g,w).
\end{equation}
Dividing \eqref{eq:last} by $t>0$ and letting $t\downarrow0$, we
obtain from \eqref{eq:derivative-at-kernel}
\[
 g(z)\ge\beta(g,w).
\]
Dividing instead by $t<0$ reverses the inequality, and letting
$t\uparrow0$ gives
\[
 g(z)\le\beta(g,w).
\]
Together, these inequalities prove \eqref{eq:beta}, and hence $w=k_z$.
\end{proof}

\subsection{Interplay of additivity and non-commutativity}
\label{subsec:algebra}

From now on, we assume that there are non-commuting linear projections
$P,Q:V\to V$, induced by factor projections $p,q:X\to X$ at $o$ as in
Lemma~\ref{lem:linear-product-decomposition}. Arguing similarly to
Lemma~\ref{lem:L-projections}, we will exploit the additivity of $\Phi$
described in Lemma~\ref{lem:additivity} to obtain strong restrictions
on these non-commuting metric projections.

We call a vector $h\in V$ a \emph{quadratic direction} for $\Phi$ if
$t\mapsto\Phi(f+th)$ is a polynomial of degree at most two for every
$f\in V$. A slight modification of the following proof shows that
every vector in the image of $PQ-QP$ is a quadratic direction for
$\Phi$. For our purposes, the following statement suffices:

\begin{lemma}\label{lem:finite-difference}
If $PQ\ne QP$, then there exists a nonzero quadratic direction
$h\in V$ for $\Phi$.
\end{lemma}

\begin{proof}
Since $PQ\ne QP$, the projection $Q$ fails to preserve at least one
of the subspaces $\im P$ and $\ker P$. Replacing $P$ by $I-P$ if
necessary, we find $v\in\im P$ such that $Q\cdot v\notin\im P$.
Hence
\[
 h:=(I-P)Q\cdot v\ne0.
\]
We claim that this $h$ is a quadratic direction for $\Phi$.
By construction,
\begin{equation}\label{eq:3eq}
 P\cdot v=v,\qquad P\cdot h=0,\qquad
 (I-P)(I-Q)\cdot v=-h.
\end{equation}
By Lemma~\ref{lem:additivity}, for any $f\in V$ and $t\in\R$ we have
\[
 \Phi(f+th)=\Phi(P\cdot f)+\Phi((I-P)\cdot f+th).
\]
Thus it suffices to prove quadraticity under the assumption
$P\cdot f=0$.

Applying the rectangle identity \eqref{eq:quadrangle} with $P$
replaced by $Q$ and $g=tv$, we obtain
\begin{equation}\label{eq:quad}
 \Phi(f+tv)+\Phi(f)
 =\Phi(f+tQ\cdot v)+\Phi(f+t(I-Q)\cdot v).
\end{equation}
Applying the $P$-splitting to the terms in \eqref{eq:quad} and using
$P\cdot f=0$ and \eqref{eq:3eq}, we obtain
\begin{align*}
 \Phi(f+tv)&=\Phi(tv)+\Phi(f),\\
 \Phi(f+tQ\cdot v)&=\Phi(tPQ\cdot v)+\Phi(f+th),\\
 \Phi(f+t(I-Q)\cdot v)&=\Phi(tP(I-Q)\cdot v)+\Phi(f-th).
\end{align*}
Substituting these expressions into \eqref{eq:quad} gives
\begin{equation}\label{eq:transfer}
 \Phi(f+th)+\Phi(f-th)-2\Phi(f)
 =\Phi(tv)-\Phi(tPQ\cdot v)-\Phi(tP(I-Q)\cdot v).
\end{equation}
The right-hand side is independent of $f\in\ker P$.

Fix $f\in\ker P$ and put $g(s)=\Phi(f+sh)$. Since $f+sh\in\ker P$,
\eqref{eq:transfer} shows that
\begin{equation}\label{eq:second}
 g(s+t)+g(s-t)-2g(s)=C(t)\qquad(s,t\in\R),
\end{equation}
where $C(t)$ does not depend on $s$.

By Lemma~\ref{lem:Phi-basic}, the function $g$ is continuous, and so
is $C$, since $C(t)=g(t)+g(-t)-2g(0)$. Integrating
\eqref{eq:second} with respect to $t$ over $[0,1]$, we obtain
\begin{equation}\label{eq:int}
 2g(s)=\int_{s-1}^{s+1}g(u)\,du-\int_0^1 C(t)\,dt.
\end{equation}
The right-hand side is $\mathcal C^1$, so $g$ is $\mathcal C^1$.
Using \eqref{eq:int} once more, we conclude that $g$ is
$\mathcal C^2$, and hence so is $C$.

Differentiating \eqref{eq:second} twice with respect to $t$ at $t=0$
gives $2g''(s)=C''(0)$. Thus, $g''$ is constant, and $g$ is a
polynomial of degree at most two. Hence, 
$h\ne0$ is a quadratic direction for $\Phi$.
\end{proof}

\subsection{Quadratic directions, splitting lines, and refinement}
\label{subsec:realization}

We now study arbitrary quadratic directions for $\Phi$.  Non-commuting factor projections do not matter anymore.

\begin{proposition}\label{prop:quadratic-direction}
Let $0\ne h\in V$ be a quadratic direction for $\Phi$.
Then $\beta(h,h)>0$ and
\begin{equation}\label{eq:quadratic-translation}
 \Phi(f+th)=\Phi(f)+2t\beta(f,h)+t^2\beta(h,h)
 \qquad(f\in V,\ t\in\R).
\end{equation}
\end{proposition}

\begin{proof}
Since $h$ is a quadratic direction, we can write
\[
 \Phi(f+th)=\Phi(f)+2c(f)t+a(f)t^2.
\]
We first show that $a(f)$ is constant. Since $\Phi$ is convex,
\[
 a(f)=\lim_{t\to+\infty}\frac{\Phi(f+th)}{t^2}
\]
is a nonnegative convex function of $f$. Moreover,
\[
 \Phi(f+th)\le\tfrac12\Phi(2f)+\tfrac12\Phi(2th)
\]
gives $a(f)\le2a(0)$ after division by $t^2$ and passage to the limit.
A convex function bounded above on an entire real vector space is
constant. Thus $a(f)=\alpha$ for every $f$.

We next show that $c$ is affine. Since $\alpha$ is independent of $f$,
\begin{align*}
 c(f)&=\lim_{t\to+\infty}\frac{\Phi(f+th)-\alpha t^2}{2t},\\
 -c(f)&=\lim_{t\to+\infty}\frac{\Phi(f-th)-\alpha t^2}{2t}
\end{align*}
are limits of convex functions of $f$. Hence both $c$ and $-c$ are
convex, so $c$ is affine. The bound
$0\le\Phi(th)\le t^2\|h\|_*^2$ from \eqref{eq:Phi-bound}
gives $c(0)=0$, so $c$ is linear. By
\eqref{eq:derivative-at-kernel},
\[
 c(k_x)=h(x)=\beta(k_x,h)\qquad(x\in X).
\]
Since the generators $k_x$ span $V$, it follows that
$c(f)=\beta(f,h)$ for every $f\in V$.

It remains to identify $\alpha$. We have $\Phi(th)=\alpha t^2$,
whereas the expansion at $f=h$ gives
\[
 \alpha(1+t)^2=\Phi(h+th)
 =\alpha+2t\beta(h,h)+\alpha t^2.
\]
Comparing the linear terms yields $\alpha=\beta(h,h)$, proving
\eqref{eq:quadratic-translation}.
Finally, since $h\ne0$ and $\beta$ is nondegenerate by
Lemma~\ref{lem:pairing}, there is $f\in V$ with $c(f)=\beta(f,h)\ne0$.
The polynomial $\Phi(f+th)$ is nonnegative for every $t\in\R$ by
Lemma~\ref{lem:Phi-basic} and has a nonzero linear coefficient.
Its quadratic coefficient $\alpha=\beta(h,h)$ must therefore be positive.
\end{proof}

A \emph{splitting line} in $X$ is a subset corresponding to
$\{z\}\times\R$ under an isometry $X\cong Z\times\R$.

The vector $k_x+th$ in \eqref{eq:quadratic-translation} need not,
a priori, represent a point of $X$. If $X$ is complete, we will show
that, for each fixed $x\in X$, the vectors $k_x+th$, $t\in\R$,
represent the points of a splitting line. Here completeness is used
only through Proposition~\ref{prop:kernel-recognition}.

\begin{lemma}\label{lem:realization}
Suppose that $X$ is complete. If $V$ contains a nonzero quadratic
direction for $\Phi$, then $X$ contains a splitting line.
\end{lemma}

\begin{proof}
Choose a nonzero quadratic direction $h$. By
Proposition~\ref{prop:quadratic-direction}, we may rescale $h$ so that
$\beta(h,h)=1$. Then \eqref{eq:quadratic-translation} becomes
\[
 \Phi(f+th)=\Phi(f)+2t\beta(f,h)+t^2
 \qquad(f\in V,\ t\in\R).
\]

Fix $x\in X$ and $t\in\R$, set $w=k_x+th$, and let $f\in V$ be arbitrary.
Applying this identity to $f$ with parameter $-t$ and to $k_x$ with
parameter $t$, we obtain
\[
 \Phi(f)+\Phi(w)
 =\Phi(f-th)+2t\beta(f,h)-t^2
   +\Phi(k_x)+2t\beta(k_x,h)+t^2.
\]
By the definition of $\Phi$,
$\Phi(f-th)\ge 2(f-th)(x)-r(x)^2$. Since $\Phi(k_x)=r(x)^2$,
we conclude that
\begin{align*}
 \Phi(f)+\Phi(w)
 &\ge 2\beta(f-th,k_x)+2t\beta(f,h)+2t\beta(k_x,h)\\
 &=2\beta(f,k_x+th)=2\beta(f,w).
\end{align*}
By Proposition~\ref{prop:kernel-recognition}, there exists a point
$z\in X$ with $w=k_z$. By injectivity of the map $y\mapsto k_y$,
this point is unique. We denote it by $\tau_t(x)$. Thus the map
$\tau:\R\times X\to X$, defined by $\tau(t,x)=\tau_t(x)$, is
well-defined and satisfies
\[
 k_{\tau_t(x)}=k_x+th.
\]
These identities imply
$\tau_s\circ\tau_t=\tau_{s+t}$ and $\tau_0=\id$.
For $x,y\in X$ and $s,t\in\R$, \eqref{eq:pairing-distance} gives
\begin{align*}
 d(\tau_t(x),\tau_s(y))^2
 &=\beta\bigl(k_{\tau_t(x)}-k_{\tau_s(y)},
               k_{\tau_t(x)}-k_{\tau_s(y)}\bigr)\\
 &=\beta\bigl(k_x-k_y+(t-s)h,\,k_x-k_y+(t-s)h\bigr).
\end{align*}
Using bilinearity of $\beta$, $\beta(h,h)=1$, and
\eqref{eq:pairing-distance} once more, we obtain
\begin{equation}\label{eq:translation}
 d(\tau_t(x),\tau_s(y))^2
 =d(x,y)^2+2(t-s)(h(x)-h(y))+(t-s)^2.
\end{equation}
In particular, each $\tau_t$ is an isometry, with inverse
$\tau_{-t}$. Also,
\[
 h(\tau_t(x))=\beta(h,k_x+th)=h(x)+t.
\]
Every orbit therefore meets $Z=h^{-1}(0)$ exactly once. By
\eqref{eq:translation}, for $z,z'\in Z$,
\[
 d(\tau_t(z),\tau_s(z'))^2=d(z,z')^2+(t-s)^2.
\]
Consequently, the map
\[
 Z\times\R\longrightarrow X,\qquad
 (z,s)\longmapsto\tau_s(z)
\]
is a surjective isometry, and each orbit is a splitting line.
\end{proof}

\begin{proof}[Proof of Theorem~\ref{thm:refinement}]
The bounded case follows from Corollary~\ref{cor:bounded}. We may therefore
assume that $X$ is complete and does not admit $\R$ as a direct factor.
Let $p,q$ be factor projections at $o$, and let $P,Q$ be the induced
linear projections on $V$. If $PQ\ne QP$,
Lemma~\ref{lem:finite-difference} gives a nonzero quadratic direction
$h\in V$ for $\Phi$. Lemma~\ref{lem:realization}
gives a line factor of $X$, contrary
to the hypothesis. Thus all induced linear projections commute, and
so do all factor projections at $o$ by
Subsection~\ref{subsec:geometric-algebraic-commutation}.
Lemma~\ref{lem:geometric-refinement} gives the common refinement.
\end{proof}

\section{The Boolean algebra of factor projections}
\label{sec:projections}

Boolean algebras of projections are a basic tool in the $L^p$-structure
theory of Banach spaces; see Cunningham~\cite{Cun60},
Behrends~\cite{Beh76}, and
\cite[Theorem~I.1.10(b) and pp.~45--47]{HWW93}. The arguments below
develop metric counterparts of this structure, including completeness
of the projection algebra, its atomic part, and product decompositions.

Let $X$ be a metric space with basepoint $o$, and write $r(x)=d(o,x)$.
Throughout this section, all factor projections at $o$ are assumed to
commute. Let $\cP$ be their set. We write $0$ for the constant map to
$o$, $1$ for the identity, and $p^{\perp}=\bar p$ for the complement of $p$.

\subsection{Boolean structure}
\label{subsec:boolean-structure}
The set $\cP$ is a Boolean algebra with
\begin{equation*}
 p\wedge q=p\circ q,\qquad
 p\vee q=(p^{\perp}\circ q^{\perp})^{\perp},\qquad
 p\le q\ \Longleftrightarrow\ p\circ q=p.
\end{equation*}
Here $p\vee q$ is the least upper bound of $p,q$, and $p\wedge q$ is
their greatest lower bound. To describe these operations geometrically, write
$X=A\times B=C\times D$, with $p,q$ projecting onto $A,C$, respectively,
and identify all factors with their fibers through $o$.
The complements $p^{\perp},q^{\perp}$ project onto the cofactors $B,D$.
Using the common refinement of Lemma~\ref{lem:geometric-refinement},
the corresponding factor fibers satisfy
\begin{equation}\label{eq:formel}
 (p\wedge q)(X)=A\cap C,\qquad
 (p\vee q)(X)=(A\cap C)\times(A\cap D)\times(B\cap C).
\end{equation}
The order relation $p\le q$ is equivalent to $A\subseteq C$.
We call $p,q$ \emph{disjoint} if $p\wedge q=0$, equivalently
$A\cap C=\{o\}$.

Changing the basepoint replaces the factor fibers and their common
refinements by the corresponding parallel fibers. Thus commutation
holds at every basepoint, and the resulting Boolean algebras are
canonically isomorphic.

An \emph{atom} of a Boolean algebra is a minimal nonzero element.
The atoms of $\cP$ are precisely the projections onto irreducible
factor fibers. Indeed, a nontrivial splitting of $p(X)$ extends to a
splitting of $X$ and gives a proper nonzero subprojection of $p$.
Conversely, any $0<q<p$ gives the nontrivial splitting
$p(X)=q(X)\times(p\wedge q^{\perp})(X)$.

\subsection{Completeness and arbitrary products}
\label{subsec:complete-projections}
The following limiting argument is the same as in the proof
of~\cite[Lemma~4.1]{FLS25} for increasing Hilbert factor fibers.
For the Banach-space counterpart, see~\cite[Theorem~I.1.10(b)
and pp.~46--47]{HWW93}.

\begin{lemma}\label{lem:complete-projections}
Suppose that $X$ is complete. Then $\cP$ is a complete Boolean algebra:
every family $(p_\alpha)$ in $\cP$ has a least upper bound
$\bigvee_\alpha p_\alpha$.

For pairwise disjoint $(e_i)_{i\in I}$, put $e=\bigvee_{i\in I}e_i$.
The canonical coordinate map $x\mapsto(e_i(x))_{i\in I}$ gives an
isometric identification
\begin{equation*}
 e(X)=\prod_{i\in I}(e_i(X),o).
\end{equation*}
\end{lemma}

\begin{proof}
Set $q_\alpha=p_\alpha^{\perp}$. We may add to the family $(p_\alpha)$
the finite joins $p_{\alpha_1}\vee\cdots\vee p_{\alpha_k}$ without
changing its set of upper bounds.

Thus, without loss of generality, we may assume that $(p_\alpha)$ is
an increasing net and $(q_\alpha)$ is a decreasing net. Set
\[
 A=\overline{\bigcup_\alpha p_\alpha(X)},\qquad
 B=\bigcap_\alpha q_\alpha(X).
\]
By completeness of $X$, as in the proof of~\cite[Lemma~4.1]{FLS25},
$(p_\alpha)$ converges pointwise to a map $p:X\to A$, and
$(q_\alpha)$ converges pointwise to a map $q:X\to B$. As in the cited
proof, continuity and completeness show that $(p,q):X\to A\times B$
is a surjective isometry.

Since $A$ is the smallest closed subset containing all $p_\alpha(X)$,
we deduce that $p$ is the least upper bound of $(p_\alpha)$. This
finishes the proof of the first statement.

The second statement follows from \eqref{eq:formel} when $|I|=2$,
and by induction for finite $I$. For infinite $I$, the construction
above shows that $e(X)$ is the closure of
\[
 \bigcup_{\substack{F\subset I\\ |F|<\infty}}
       \prod_{i\in F}(e_i(X),o)\subset X\,,
\]
finishing the proof of the second claim.
\end{proof}

\subsection{Atomic and atomless factors}
\label{subsec:atomic-factor}
A Boolean algebra is \emph{atomic} if every nonzero element contains an
atom, and \emph{atomless} if it has no atoms. Every complete Boolean
algebra $\cB$ has a unique decomposition into an atomic and an atomless
factor; see~\cite[Corollary~2, p.~227]{GH09}.
Together with Lemma~\ref{lem:complete-projections} and the correspondence
between atoms and irreducible factor fibers, this gives the following
metric decomposition.

\begin{proposition}\label{prop:atomic-decomposition}
Suppose that $X$ is complete, and let $(e_i)_{i\in I}$ be all atoms
of $\cP$. Put $a=\bigvee_{i\in I}e_i$. Then the splitting
\begin{equation*}
 X=X_{\mathrm{at}}\times X_{\mathrm{na}},
 \qquad X_{\mathrm{at}}=a(X),\quad X_{\mathrm{na}}=a^{\perp}(X),
\end{equation*}
is canonical for the fixed basepoint $o$. The factor $X_{\mathrm{at}}$
is the pointed product of the irreducible factor fibers $e_i(X)$,
while $X_{\mathrm{na}}$ has no irreducible factor.
\end{proposition}

\subsection{Compactness of the projection algebra}
We continue to assume that all factor projections commute and assume
in addition that $X$ is proper.

By Arzel\`a--Ascoli, every sequence $(p_i,q_i)$ of pairs of
complementary projections has a subsequence converging uniformly on
bounded sets. By continuity, the limiting $1$-Lipschitz maps $p,q$
define an isometric embedding $(p,q):X\to p(X)\times q(X)$.
The limiting pair $(p,q)$ is again a pair of complementary
projections. This gives the following:

\begin{lemma}\label{lem:proper-projections}
If $X$ is proper, then $\cP$ is compact and metrizable for the topology
of uniform convergence on bounded sets, and its Boolean operations are
continuous.
\end{lemma}

By the theorem of Papert Strauss~\cite{PS68}, every compact Hausdorff
topological Boolean algebra is a powerset algebra and hence atomic;
see also \cite[Theorem~4.1]{BH16}. Together with
Proposition~\ref{prop:atomic-decomposition}, this gives the following.

\begin{proposition}\label{prop:proper-commuting-decomposition}
Let $X$ be a proper metric space whose factor projections at a
basepoint $o$ commute. Then the atoms $(e_i)_{i\in I}$ of $\cP$ give
an isometry
\begin{equation}\label{eq:proper-atomic-product}
 X\longrightarrow\prod_{i\in I}(e_i(X),o),
 \qquad x\longmapsto(e_i(x))_{i\in I},
\end{equation}
where $I$ is finite or countable and the factors are nontrivial,
irreducible, and proper. These factor fibers are uniquely determined
up to permutation.
\end{proposition}

\begin{proof}
By Lemma~\ref{lem:proper-projections} and the preceding discussion, $\cP$
is compact, metrizable, and atomic. Its atoms form a discrete subspace:
for each atom $e$, the continuous map $p\mapsto p\wedge e$ takes values
in the discrete set $\{0,e\}$, so $\{p\in\cP:e\le p\}$ is open and
contains no other atom. Since $\cP$ is second countable, it has at most
countably many atoms. Proposition~\ref{prop:atomic-decomposition}
therefore gives \eqref{eq:proper-atomic-product}, and the atom fibers are
uniquely determined. They are closed subsets of $X$, hence proper.
The one-point case corresponds to the empty product.
\end{proof}

\section{Canonical decompositions and finite roots}
\label{sec:decompositions}

We now combine the commutation theorems with the projection algebra
of Section~\ref{sec:projections} to prove
Theorems~\ref{thm:proper} and~\ref{thm:roots}.

\subsection{The canonical decomposition}
The atomic decomposition from Section~\ref{sec:projections} applies
to the cofactor of the maximal Hilbert space factor.

\begin{theorem}\label{thm:canonical-decomposition}
Every complete metric space $X$ admits a decomposition
\begin{equation*}
 X\cong \mathbb H\times\prod_{i\in I}(X_i,o_i)\times X_{\mathrm{na}},
\end{equation*}
where $\mathbb H$ is a Hilbert space, the $X_i$ are nontrivial irreducible
complete metric spaces not isometric to $\R$, and $X_{\mathrm{na}}$
has no nontrivial irreducible factor. For each basepoint, all these
factor fibers are uniquely determined, up to permutation of the $X_i$.
\end{theorem}

\begin{proof}
By~\cite[Theorem~1.1]{FLS25}, split off the canonical Hilbert factor,
writing $X=\mathbb H\times Y$. The complete space $Y$ has no splitting lines,
so its factor projections commute by Theorem~\ref{thm:refinement}.
Proposition~\ref{prop:atomic-decomposition} gives the remaining decomposition.

A product of spaces without splitting lines has no splitting lines,
since the projection of a splitting line onto each factor would be a
point by~\cite[Proposition~3.2]{FLS25}. Thus, in any decomposition of the
stated form, $\mathbb H$ is the canonical Hilbert factor, and
Proposition~\ref{prop:atomic-decomposition} gives uniqueness of the
remaining factor fibers.
\end{proof}

We can now finish the proof of Theorem~\ref{thm:proper}:

\begin{proof}[Proof of Theorem~\ref{thm:proper}]
Consider the decomposition of $X$ provided by
Theorem~\ref{thm:canonical-decomposition}. All factor fibers are closed
in $X$, hence proper. Thus the Hilbert factor $\mathbb H$ is finite-dimensional
and hence Euclidean. Apply
Proposition~\ref{prop:proper-commuting-decomposition} to the cofactor of
$\mathbb H$, whose factor projections commute by Theorem~\ref{thm:refinement}.
It follows that $I$ is finite or countable and $X_{\mathrm{na}}$ is a
one-point space. Uniqueness follows from
Theorem~\ref{thm:canonical-decomposition}.
\end{proof}

\subsection{Finite roots of proper spaces}
\label{subsec:proper-roots}

We now prove uniqueness of finite roots of proper spaces, providing an
affirmative answer to Ulam's question in this setting. Some care is
needed, since we do not assume that our isometries preserve basepoints.

\begin{proof}[Proof of Theorem~\ref{thm:roots}]
The case $n=1$ is immediate, so assume $n\ge2$. Let $A,B$ be proper
metric spaces with $A^n\cong B^n$. Consider the canonical decompositions
\[
 A=\R^k\times A_0,\qquad B=\R^l\times B_0,
\]
where $A_0,B_0$ have no splitting lines. The canonical Hilbert factors
of $A^n$ and $B^n$ are $\R^{nk}$ and $\R^{nl}$,
by~\cite[Theorem~1.1 and Proposition~3.2]{FLS25}.
Hence $k=l$, and uniqueness of the Hilbert splitting gives
\[
 A_0^n\cong B_0^n.
\]
Thus it suffices to consider the case in which $A$ and $B$ have no
splitting lines. By Theorem~\ref{thm:proper}, choose basepoints $a\in A$,
$b\in B$ and decompositions
\[
 (A,a)=\prod_{i\in I}(A_i,a_i),\qquad
 (B,b)=\prod_{j\in J}(B_j,b_j)
\]
into irreducible factors, where $I,J$ are finite or countable. Fix an
isometry $F:A^n\to B^n$. Compare the images under $F$ of the factor
fibers through $(a,\ldots,a)$ with the target factor fibers through
$F(a,\ldots,a)$. Uniqueness gives a bijection
\[
 \sigma:I\times\{1,\ldots,n\}\longrightarrow
        J\times\{1,\ldots,n\}.
\]
If $\sigma(i,k)=(j,l)$, reading the corresponding target coordinate
of the restriction of $F$ gives an isometry
$\phi_{i,k}:A_i\to B_j$. Put
$c_{i,k}=d(\phi_{i,k}(a_i),b_j)$. Then
\begin{equation}\label{eq:root-displacements}
 \sum_{i\in I}\sum_{k=1}^n c_{i,k}^2
 =d\bigl(F(a,\ldots,a),(b,\ldots,b)\bigr)^2<\infty.
\end{equation}

The bijection $\sigma$ defines an $n$-regular bipartite multigraph
with vertex sets $I,J$: each equality $\sigma(i,k)=(j,l)$ contributes
one edge from $i$ to $j$. Since $n$ is positive and finite, this graph
admits a perfect matching by the locally finite extension of Hall's
marriage theorem~\cite{Hall35}; see also~\cite[Theorem~4.2]{Lac02},
which allows multiple edges.

Let $\tau:I\to J$ be the resulting bijection. For each $i$, choose the
pair $(k,l)$ corresponding to its matched edge, so that
$\sigma(i,k)=(\tau(i),l)$, and use the corresponding isometry
$\phi_i:A_i\to B_{\tau(i)}$.
By \eqref{eq:root-displacements},
\[
 \sum_i d(\phi_i(a_i),b_{\tau(i)})^2<\infty.
\]
The inverse coordinate isometries have the same displacements.
Thus the coordinatewise product of the $\phi_i$ is a surjective
isometry $A\to B$.
\end{proof}

\section{Cardinal algebras and pointed finite roots}
\label{sec:pointed-roots}

We prove Theorem~\ref{thm:roots3} using Tarski's cancellation theorem
for cardinal algebras. 
Throughout this section, $\cong_*$ denotes a basepoint-preserving
surjective isometry.

\subsection{Cardinal algebras}
We recall the axioms in the form used in~\cite[Section~2(A)]{KM16}.
A \emph{cardinal algebra} is a commutative monoid $(\mathcal A,+,0)$
with an operation assigning an element $\sum_{k\geq0}a_k$ to every
sequence $(a_k)$ in $\mathcal A$, subject to the following  axioms:
\begin{align*}
 \sum_{k\geq0}a_k&=a_0+\sum_{k\geq0}a_{k+1},\\
 \sum_{k\geq0}(a_k+b_k)&=\sum_{k\geq0}a_k+\sum_{k\geq0}b_k.
\end{align*}
The \emph{refinement axiom}: whenever
$a+b=\sum_{k\geq0}c_k$, there are sequences $(a_k),(b_k)$ such that
\[
 a=\sum_{k\geq0}a_k,\qquad b=\sum_{k\geq0}b_k,\qquad c_k=a_k+b_k.
\]
The \emph{remainder axiom}: whenever $a_k=b_k+a_{k+1}$ for
all $k\geq0$, there is $c\in\mathcal A$ such that
\[
 a_k=c+\sum_{j\geq k}b_j\qquad(k\geq0).
\]
These axioms are equivalent to Tarski's original definition;
see~\cite[Remark~2.1]{KM16}.

We use the following cancellation theorem of Tarski; see
\cite[Theorem~2.35]{Tar49} or \cite[Theorem~2.2(8)]{KM16}.
Here $na$ denotes the sum of $n$ copies of $a$.
\begin{theorem}\label{thm:tarski-cancellation}
In a cardinal algebra, $na=nb$ implies $a=b$ for every positive
integer $n$.
\end{theorem}
%This is~\cite[Theorem~2.35]{Tar49}; see also
%\cite[Theorem~2.2(8)]{KM16}. Here $na$ denotes the sum of $n$ copies
%of $a$. This is cancellation of a common finite multiplicity, not
%cancellation of an arbitrary common summand.

\subsection{The cardinal algebra of pointed metric spaces}
Fix an infinite cardinal $\kappa$ with $\kappa^{\aleph_0}=\kappa$. Let $\mathcal A_\kappa$ be the set
of pointed isometry classes of complete metric spaces of cardinality at
most $\kappa$ that do not admit $\R$ as a direct factor. Define  $0\in \mathcal A_\kappa$   as the pointed one-point space, and 
% This is a
%set: every such space has cardinality at most $\kappa^{\aleph_0}$,
%since its points are determined by convergent sequences in a dense
%subset of cardinality at most $\kappa$.
define addition and countable summation by pointed products:
\begin{align*}
 [X,x]+[Y,y]&=[X\times Y,(x,y)],\\
 \sum_{k\geq0}[X_k,x_k]
 &=\left[\prod_{k\geq0}(X_k,x_k),(x_k)_{k\geq0}\right].
\end{align*}
%The zero element is the pointed one-point space.
% These operations are
%well-defined: coordinatewise pointed isometries preserve the
%square-summability condition defining the product.

Complete spaces are closed under pointed products, and their factor
fibers are complete. A countable pointed product of spaces of cardinality 
at most $\kappa$ has cardinality at most $\kappa$ by our choice of $\kappa$.  Factor fibers also have cardinality at most $\kappa$.
% since
%they are continuous images under the factor projections.
Finally, a pointed product of spaces without line factors has no line
factor by~\cite[Proposition~3.2]{FLS25}.
% the projection of a
%splitting line onto each coordinate factor is either a splitting line
%or a point. All these projections would therefore be constant, which
%is impossible for a line in the product. 
Thus the operations above are well-defined 
 in $\mathcal A_\kappa$.

\begin{proposition}\label{prop:metric-cardinal-algebra}
The pointed products make $\mathcal A_\kappa$ a cardinal algebra.
\end{proposition}

\begin{proof}
The monoid and summation identities follow by reindexing the factors.
We verify refinement and the remainder axiom using
Theorem~\ref{thm:refinement} and Lemma~\ref{lem:complete-projections}.

For refinement, realize $a+b=\sum_{k\geq0}c_k$ by two pointed
product decompositions of a space $(Z,o)$, and let $p$ and $e_k$ be the
factor projections corresponding to $a$ and $c_k$, respectively.
The $e_k$ are pairwise disjoint, with $\bigvee_k e_k=1$. Since the
factor projections commute and form a complete Boolean algebra,
\[
 p=\bigvee_k(p\wedge e_k),\qquad
 p^\perp=\bigvee_k(p^\perp\wedge e_k),\qquad
 e_k=(p\wedge e_k)\vee(p^\perp\wedge e_k).
\]
All these joins are disjoint. Taking the pointed isometry classes of
the corresponding factor fibers gives the required refinement.

For the remainder axiom, realize $a_k=b_k+a_{k+1}$ by a nested
sequence of factor fibers $X_k$ through $o$ in a common pointed space
$X_0$, with $[X_k,o]=a_k$. Let $p_k$ be the corresponding projections
and put
\[
 p_\infty=\bigwedge_k p_k,\qquad
 r_k=p_k\wedge p_{k+1}^\perp.
\]
The factor fiber $r_k(X_0)$ represents $b_k$, and the complete Boolean
algebra gives the disjoint-join identity
\[
 p_k=p_\infty\vee\bigvee_{j\geq k}r_j.
\]
By Lemma~\ref{lem:complete-projections},
$c=[p_\infty(X_0),o]\in\mathcal A_\kappa$ satisfies
$a_k=c+\sum_{j\geq k}b_j$, as required.
\end{proof}

\subsection{Proof of Theorem~\ref{thm:roots3}}
\begin{proof}[Proof of Theorem~\ref{thm:roots3}]
First suppose that $A$ and $B$ have no line factors. Choose an infinite
cardinal $\kappa$ with $\kappa^{\N}=\kappa$ bounding their
cardinalities. Then $n[A,a]=n[B,b]$ in $\mathcal A_\kappa$, so
Theorem~\ref{thm:tarski-cancellation} gives $(A,a)\cong_*(B,b)$.

In general, split off the canonical Hilbert factors
$(A,a)\cong_*(\mathbb H_A,0)\times(A_0,a_0)$ and
$(B,b)\cong_*(\mathbb H_B,0)\times(B_0,b_0)$.
As in the first paragraph of the proof of Theorem~\ref{thm:roots},
the given pointed isometry of the $n$-th powers identifies the Hilbert
factor fibers and their line-free complements. Thus
$(\mathbb H_A^n,0)\cong_*(\mathbb H_B^n,0)$ and
$(A_0^n,(a_0,\ldots,a_0))\cong_*(B_0^n,(b_0,\ldots,b_0))$.
The latter gives $(A_0,a_0)\cong_*(B_0,b_0)$ by the first case;
the former gives $\dim \mathbb H_A=\dim \mathbb H_B$. Taking products finishes
the proof.
\end{proof}

\section{Measurable matchings and nonunique square roots}
\label{sec:counterexample}

\subsection{Laczkovich's graph} 

  All measures below are completed, and
measurable maps and equalities are understood modulo null sets when
appropriate.

Fix an irrational $\alpha\in(0,1)$ and consider the measure-preserving
involutions of the circle  $\T=\R/\Z$ given by
\[
 s(x)=-x,\qquad t(x)=\alpha-x.
\]
Their composition $t\circ s$ is the ergodic irrational rotation by
$\alpha$. Remove the countable invariant set
\[
 D=\{x\in\T:2x=n\alpha\pmod1\text{ for some }n\in\Z\},
\]
and put $\Omega=\T\setminus D$, with Lebesgue probability measure
$\lambda$. The action of the group generated by $s,t$ is free on
$\Omega$. Set
\[
 V_0=\Omega/\langle s\rangle,\qquad
 V_1=\Omega/\langle t\rangle,\qquad
 \mu_i=(\pi_i)_*\lambda,
\]
where $\pi_i:\Omega\to V_i$ are the quotient maps. These are standard
atomless probability spaces.

Define a $2$-regular bipartite graph $G$ with vertex set
$V_0\sqcup V_1$ and edge set $\Omega$. The edge $x\in\Omega$
joins $\pi_0(x)$ to $\pi_1(x)$.
 Every component of the graph  is a bi-infinite path, since $\langle s,t\rangle$ acts freely on $\Omega$. 
Below, when referring to components of the graph we will only mean the vertices of a component.  We define as $d_G$ the distance on $V_0\cup V_1$ obtained by assigning length $1$ to each edge (thus vertices lie at  a finite distance if and only if they lie in the same component).

Each component has two perfect matchings, so a set-theoretic perfect
matching exists. However, a measurable perfect matching, thus a measurable bijective map  $f:V_0'\to V_1'$ between full-measure subsets $V_i'\subset V_i$ such that $v$ and $f(v)$  are adjacent for all $v \in V_0'$, does not exist, as was observed by Laczkovich in \cite{Lac88}.  Indeed, such a  measurable matching would define a measurable subset  $M$ of $\Omega$ such that
\[
 s(M)=\Omega\setminus M=t(M)
 \quad\text{modulo null sets}.
\]
We have $\lambda(M)=1/2$, while $M$ is invariant under the ergodic map
$t\circ s$, a contradiction.

The integrability assumption in the following strengthening of
Laczkovich's obstruction is essential. Indeed, Dye's lemma for ergodic
measured equivalence relations yields measurable component-preserving
bijections without this assumption; see~\cite[Lemma~3.1.4(2)]{Pop07}.

\begin{lemma}\label{lem:integrable-matching}
There is no  measurable bijection modulo null sets
$\theta:V_0\to V_1$ such that $v$ and $\theta(v)$ belong to the same
component almost everywhere and
\[
 \int_{V_0}d_G(v,\theta(v))\,d\mu_0(v)<\infty.
\]
\end{lemma}

\begin{proof}
Suppose that $\theta$ exists. The graph saturation of a null set is
a null set, since it is obtained by countably many applications of the
measure-preserving involutions. After discarding such a saturation,
we may work with full-measure  subsets of $V_0,V_1$ on which $\theta$ is a
component-preserving bijection. For  simplicity, we do not distinguish between $V_i$ and those subsets.

For $v\in V_0$, let $P_v$ be the unique path from $v$ to $\theta(v)$,
and put $L(v)=d_G(v,\theta(v))$. The function $L$ and the first edge
$e(v)$ of $P_v$ depend measurably on $v$. The map $e$  is injective and selects one
edge from each pair of edges adjacent to a given vertex. Hence, for every measurable $D\subset V_0$,
\[
 \lambda(e(D))=\tfrac12\mu_0(D).
\]
For $k\geq1$, let $E_k$ be the set of edges occurring in position $k$
in some $P_v$. The $k$-th edge is obtained from $e(v)$, for
$L(v)\geq k$, by a fixed measure-preserving word in $s,t$. Therefore
\[
 \lambda(E_k)=\tfrac12\mu_0\{L\geq k\}.
\]
The measurable function $N(x)=\sum_{k\geq1}\one_{E_k}(x)$ counts the
paths crossing $x$. Consequently,
\begin{equation}\label{eq:path-crossings}
 \int_\Omega N\,d\lambda
 =\tfrac12\sum_{k\geq1}\mu_0\{L\geq k\}
 =\tfrac12\int_{V_0}L\,d\mu_0<\infty.
\end{equation}
Thus $N$ is finite almost everywhere. Removing a saturated null set as above, we may assume that $N$ is finite everywhere. Define
\[
 M=\{x\in\Omega:N(x)\text{ is odd}\}.
\]
At each vertex, exactly one of the paths $P_v$ starts or ends, while
every path passing through the vertex uses both incident edges.
Therefore the crossing counts of the two incident edges have opposite
parities. Thus, 
$M$ contains exactly one edge incident to each vertex. It is thus a
measurable perfect matching, contrary to
Laczkovich's obstruction above.
\end{proof}

\subsection{Lifts of components}
Represent $x\in\Omega$ by $\widetilde x\in (0,1)$ and change the length of  the edge $x$ of the graph $G$ from the length $1$  in $(G,d_G)$ to 
\[
 \ell(x)=1+\widetilde x \in (1,2)
\]
Equip the vertex set of each component $C$ with the resulting weighted
path metric $d_C$. By construction $d_C\geq d_G$. The subset $C$ is
isometric to a closed discrete subset of $\R$, with all consecutive
gaps in $(1,2)$. The edges connecting adjacent vertices in $C$   contain an orbit of the  irrational rotation, which is  dense in the unit circle. Hence,
\begin{equation}\label{eq:component-separation}
 \inf\{d_C(u,v):u,v\in C,\ u\ne v\}=1.
\end{equation}
Two vertices are adjacent
exactly if their distance is in $(1,2)$.

Different edges of $G$ have different lengths.
Therefore,  the existence of an  isometry between components $C_1$ and $C_2$ implies that $C_1=C_2$ and the isometry is the identity.
Moreover,  any  bijective similarity between components has to be an isometry by
\eqref{eq:component-separation}. We have thus verified, that the only bijective similarity between  components is the identity.

For a vertex $v$, write $C(v)$ for its component.  There is a unique isometric embedding $\kappa_v:C(v)\to\R$ with $\kappa_v(v)=0$, such that 
the closest point to $v$ in $C(v)\setminus \{v\}$ is mapped to a positive real number.
Write
\[
 S_v=\kappa_v(C(v))=\{a_n(v):n\in\Z\},\qquad a_0(v)=0,
\]
where the coordinates are indexed in increasing order. By construction,
\begin{equation}\label{eq:coordinate-gaps}
	|a_n(v)|\leq  2|n|,\qquad
	1< a_{n+1}(v)-a_n(v)<2.
\end{equation}
Each $a_n$ is
measurable on $V_i$, since it can be expressed in terms of the involutions $s$ and $t$.  

\subsection{The  spaces and their squares}
Consider the Hilbert spaces $\mathbb H_i=L^2(V_i,\mu_i;\R)$ and 
\begin{equation}\label{eq:counterexample-spaces}
 X_i=\{f\in \mathbb H_i:f(v)\in S_v\text{ for almost every }v\},
 \qquad A=X_0,\quad B=X_1.
\end{equation}
Since all $S_v$ are closed in $\R$, each  $X_i$ is closed in $\mathbb H_i$.
Thus $A,B$ are complete and separable. The linear span of $X_i$ is  dense
in $\mathbb H_i$, since $a_1\one_E\in X_i$ for every
measurable $E\subset V_i$, and $a_1$ is positive and bounded.

We are going to identify  $A^2$ and $B^2$ with a space $Z$ of sections over $\Omega$, thus establishing an isometry between $A^2$ and $B^2$.   For $x\in\Omega$, let $C(x)$ be the
component containing the endpoints of $x$.
 Let $Z$ consist of measurable sections
$z:\Omega\to V_0\sqcup V_1$ such that $z(x)\in C(x)$ almost everywhere
and
\[
 2\int_\Omega d_{C(x)}(z(x),\pi_0(x))^2\,d\lambda(x)<\infty.
\]
Since
\[
d_{C(x)}(\pi_0(x),\pi_1(x))=\ell(x)<2,
\]
the same space $Z$ is obtained if $\pi_0$ is replaced by $\pi_1$.
Equip $Z$ with the metric 
\[
 d_Z(z,z')
 =\sqrt{2\int_\Omega d_{C(x)}(z(x),z'(x))^2\,d\lambda(x)}.
\]

For $i=0,1$, order the two edges incident to $v\in V_i$ measurably,
writing them as $e_i^1(v),e_i^2(v)$. For every nonnegative measurable
$q$ on $\Omega$,
\begin{equation}\label{eq:edge-incidence}
 2\int_\Omega q\,d\lambda
 =\int_{V_i}\bigl(q(e_i^1(v))+q(e_i^2(v))\bigr)\,d\mu_i(v).
\end{equation}
Given $(f_1,f_2)\in X_i\times X_i$, define a section of $Z$ by
\[
 z(e_i^j(v))=\kappa_v^{-1}(f_j(v)),\qquad j=1,2.
\]
This section $z$ is measurable by construction.  The  map
$(f_1,f_2)\mapsto z$ is bijective with inverse   
$f_j(v)=\kappa_v(z(e_i^j(v)))$. 
The identity~\eqref{eq:edge-incidence}, applied to squared distances
in $C(x)$, shows that this bijective map is an  isometry
$X_i\times X_i\cong Z$. Consequently,
\begin{equation}\label{eq:counterexample-squares}
 A\times A\cong Z\cong B\times B.
\end{equation}

\subsection{$A$ and $B$ are not isometric} 
We are  going to analyze possible isometries between $A$ and $B$ following a similar strategy as in  the proof of~\cite[Theorem~E]{Len25}.

First we are going to determine the factor projections of $X_i$.
Every measurable subset $D\subset V_i$ defines a factor
projection $p_D:X_i\to X_i$ based at the zero function,
given by
\[
p_D(f)=\one_D f.
\]
Its complementary projection is $p_{V_i\setminus D}$.

We claim that these are all factor projections of $X_i$
based at zero. 

Indeed, since the linear span of $X_i$
is dense in $\mathbb H_i$, the uniqueness of Hilbert-space
embeddings \cite[Theorem~11.4]{WW75}, applied to the
product decomposition, shows that every factor
projection $p$ extends to an orthogonal projection
$P:\mathbb H_i\to\mathbb H_i$.

Since the fibers $S_v$ are discrete, $X_i$ contains
no nonconstant geodesic segments and hence has
no line factors. By Theorem~\ref{thm:refinement},
$p$ commutes with every $p_D$.
Consequently, $P$ commutes with multiplication
by $\one_D$ for every measurable $D\subset V_i$.
It is a classical fact that an orthogonal projection
on $\mathbb H_i=L^2(V_i,\mu_i)$ commuting with all such
multiplication operators is itself multiplication
by a characteristic function~\cite[Example~5.1.6]{KR97}.
Thus $p=p_D$ for some measurable $D\subset V_i$, as claimed.

As a consequence, the factor projection of $X_i$ based at a point $h\in X_i$ has the form
\begin{equation} \label{eq:pdf}
f	\mapsto  \one _D f  + \one _{V_i\setminus D}h
\end{equation}

Assume  that there exists a surjective isometry $F:X_0\to X_1$. Put $h=F(0)$. 
By the Hilbert-space extension theorem~\cite[Theorem~11.4]{WW75},
there exists a surjective linear isometry  $U:\mathbb H_0\to\mathbb H_1$ such that  $F(f)=h+U\cdot f$.

Conjugation by $F$ identifies the factor projections of $X_0$ based at $0$ with those of $X_1$ based at $h$.
 Due to \eqref{eq:pdf}, $F$ induces  
an isomorphism $\eta$ between the measure algebras (that is the algebra of measurable subsets modulo null sets) of $V_0$ and $V_1$, such that
\[U\cdot \one _D f=\one _{\eta (D)} (U\cdot f),\]
for every measurable subset $D\subset V_0$ and every $f\in \mathbb H_0$.

The realization theorem
for standard measure algebras~\cite[343B, 343F]{Fre02}
provides a full-measure subsets $V_i'\subset V_i$ and a measurable bijection
$\phi:V_1'\to V_0'$
with $\eta(D)=\phi^{-1}(D)$ modulo null sets.

Put 
$w=U\cdot\one$.  Hence $U\cdot \one _D =w\one _{\phi^{-1} (D)}$. 
% In particular, 
%\[\int _{\phi^{-1} (D)} w^2 \, d\mu_1= \mu _0 (D) . \]
In other words, for all $f=\one _D$ we have
%We claim that for all $f\in \mathbb H_0$,
\begin{equation}  \label{eq:density}
U\cdot f =w \cdot (f\circ \phi).
\end{equation}
Linearity of this equation and density of simple functions in $\mathbb H_0$ implies that \eqref{eq:density} holds true for all $f\in \mathbb H_0$. 

Applying \eqref{eq:density} to the functions $f=a_n$ (whose values at $v$  enumerate  $S_v$), we deduce 
\begin{equation} \label{eq:verylast}
	S_v\supset h(v)+ w(v)\cdot S_{\phi (v)}, 
	\end{equation} 
for almost all $v\in V_1$.   Applying the same argument to $F^{-1}$ instead of $F$ we deduce that the inclusion in \eqref{eq:verylast} is an equality.

Since the only similarities between the components $S_v$ and $S_{\phi (v)}$  are isometries,   we deduce  $|w(v)|=1$, for almost all $v \in V_1$.  Since such isometries only exist if  the components $C(v)$ and $C(\phi (v))$ coincide, we deduce $\phi (v)\in C(v)$ for almost all $v\in V_1$.

Since the only isometry between $S_{\phi (v)}$ and $S_v$  sends $0\in S_{\phi (v)}$  to $\pm  d_C(v,\phi (v)) \in  S_v$ we deduce that $|h(v)|= d_C(v, \phi (v))$.

Since $d_C\geq d_G$, we deduce $d_G(\phi (v),v)) \leq |h(v)|$ for almost all $v$.    As $h\in \mathbb H_1$ and $\mu_1$ is a
probability measure,
\[
 \int_{V_1}d_G(v,\phi(v))\,d\mu_1(v)
 \leq\int_{V_1}|h(v)|\,d\mu_1(v)
 \leq\|h\|_{L^2}<\infty.
\]

After exchanging the roles of $V_1$ and $V_0$,  $\phi$ 
 contradicts Lemma~\ref{lem:integrable-matching}.  This contradiction shows that no surjective isometry $F:X_0\to X_1$ can exist. Thus,  $A\not\cong B$.

 Together with~\eqref{eq:counterexample-squares} this finishes the proof of Theorem \ref{thm:roots2}.

\end{document}